\documentclass[11pt]{article}

\usepackage{amssymb}
\usepackage{amsthm}
\usepackage{amsmath}
\usepackage{latexsym}
\usepackage{tikz}

\newtheorem{theorem}{Theorem}[section]
\newtheorem{lemma}[theorem]{Lemma}
\newtheorem{cor}[theorem]{Corollary}

\makeatletter \@addtoreset{equation}{section} \makeatother

\newcommand{\pp}{\partial_+}
\newcommand{\pn}{\partial_-}
\def\tilde{\widetilde}
\def\hatN{\widehat N}
\newcommand{\si}{\sigma}
\newcommand{\RR}{\mathbb{R}}
\newcommand{\beq}{\begin{equation}}
\newcommand{\eeq}{\end{equation}}
\newcommand{\eps}{\varepsilon}
\def\TS{\textstyle}
\def\com#1{\quad{\textrm{#1}}\quad}
\def\eq#1{(\ref{#1})}
\def\nn{\nonumber}
\def\ol{\overline}
\def\ora#1{{\overrightarrow{#1}}\!}
\def\ola#1{{\overleftarrow{#1}}\!}
\def\mc{\mathcal}
\def\bint{\mathop{\mathrel{\nwarrow}\!\!\!\!\!\!\mathrel{\int}}}
\def\fint{\mathop{\mathrel{\nearrow}\!\!\!\!\!\!\mathrel{\int}}}
\def\Lr{{\ora{\mc L}}}
\def\Ll{{\ola{\mc L}}}
\def\x#1{$#1\times#1$}
\def\dbyd#1{\TS{\frac{\partial}{\partial#1}}}

\begin{document}

\title{Shock formation in the compressible Euler equations and related
  systems}

\author{Geng Chen\thanks{Department of Mathematics, Pennsylvania State
University, University Park, PA 16802 ({\tt chen@math.psu.edu}).}
\and Robin Young\thanks{Department of Mathematics and Statistics,
University of Massachusetts, Amherst, MA 01003 ({\tt
young@math.umass.edu}).  Supported in part by NSF Applied
Mathematics Grant Number DMS-0908190.}
\and Qingtian Zhang\thanks{Department of Mathematics,
Pennsylvania State University, University Park, PA 16802 ({\tt zhang\_q@math.psu.edu}).}
}

\date{}
\maketitle

\begin{abstract}
  We prove shock formation results for the compressible
  Euler equations and related systems of conservation laws in one
  space dimension, or three dimensions with spherical
  symmetry.  We establish an $L^\infty$ bound for $C^1$
  solutions of the one-D Euler equations, and use this to improve recent
  shock formation results of the authors.  We prove analogous shock
  formation results for one-D MHD with orthogonal magnetic field, and
  for compressible flow in a variable area duct, which has as a
  special case spherically symmetric three dimensional flow on the
  exterior of a ball.
\end{abstract}

2010\textit{\ Mathematical Subject Classification:} 76N15, 35L65, 35L67.


\section{Introduction}

One of the defining features of systems of hyperbolic conservation
laws,
\beq
  u_t+f(u)_x=0, \quad (x,t)\in\RR\times\RR^+,\ u\in\RR^n,
\label{CL}
\eeq in the absence of regularizing dissipative effects such as
viscosity and heat loss, is the blowup of gradients and subsequent
formation of shock waves.  
Once a shock forms, discontinuous weak
solutions must be considered, and the corresponding loss of regularity
makes the systems hard to analyze.

We are concerned with the conditions under which a shock forms in a
classical $C^1$ solution of \eq{CL}.  For pairs of equations ($n=2$),
Lax showed in~\cite{lax2,lax3} that shocks always form for data which
is anywhere compressive.  In work initiated by John, this was extended
to larger systems, given restrictions on the size of the initial
data~\cite{Fritz John,Li daqian,Liu1}.  In particular, all waves in
those solutions are assumed to be weak.  One of our goals is to remove
these severe restrictions on the initial data, which includes allowing
for large waves in the solution.  We restrict our attention to systems
arising from physical problems because of the variety of unstable
phenomena which can affect general \x3 systems when the data is
allowed to be large~\cite{Jens,JY,Yblowup1,Ynonuniq,szY}.

Recently, the authors have obtained large data results for the \x3
compressible Euler equations~\cite{G3,G4}.  In all of these works,
\eq{CL} is differentiated and manipulated to get a system of PDEs with
quadratic nonlinearities for the evolution of gradients along
characteristics.  For two equations, these PDEs decouple to yield a
pair of Riccati type ODEs along characteristics; the quadratic
nonlinearity then implies that the derivatives must blow up in finite
time~\cite{lax2,lax3}.  For general systems, if the data has small
total variation, the nonlinear waves separate in finite time, after
which they are essentially decoupled and genuine nonlinearity in each
field causes gradients to blow up by the same mechanism.

In~\cite{G3,G4}, we consider the compressible Euler equations, and
make a series of nonlinear changes of variables to get an uncoupled
pair of quadratic ODEs for the genuinely nonlinear fields, which
exhibit the same blowup mechanism.  Although there are no restrictions
on the size of the initial data, the results of~\cite{G3,G4} rely on
assumed global bounds for the state variables.  In this paper, we
first remove this restriction by obtaining an \emph{a priori} estimate
on the state variables for $C^1$ solutions of the Euler equations,
yielding a condition on only the initial data which guarantees blowup
of gradients.  See \cite{LinVong} for a description of attempts to
obtain a stronger $L^\infty$ bound and global existence which were
unsuccessful.

We then extend the method to larger systems of physical interest.  We
show similar gradient blowup results for one-dimensional
magnetohydrodynamics (MHD) with orthogonal magnetic field, and
compressible flow in a variable area duct.  As a special case, we
obtain a gradient blowup result for spherically symmetric
three-dimensional Euler flow on the exterior of a ball.  These blowup
results again rely on global bounds for the state variables, but in
view of our obtained bounds for the compressible Euler equations, we
expect that this is just a technical assumption.  Indeed, we note that
all known cases of blowup of state variables occur in non-physical
systems~\cite{Jens,JY,Yblowup1,Ynonuniq,szY}.

The mechanism for gradient blowup and shock formation is the
steepening of compressions due to genuine nonlinearity: if the
wavespeed is larger behind the wave, it will eventually break.  For
\x2 systems (or diagonal systems such as chromatography), the
different wave families weakly decouple and all compressions will
collapse into shocks.  However, in \x3 systems, and especially for
large data, the wave interactions between different families influence
this process and can delay the onset of
shocks~\cite{linliuyang,liufg,young blake 1}.  In the context of
Euler, the contacts act as partial reflectors of waves and with care
one can set up solutions in which compressions are cancelled before
they focus~\cite{young blake 1,G1}.

Our conditions which guarantee blowup are expressed in terms of global
bounds for the variables, so we must ensure that the state does not
leave some compact set.  In particular, for a $\gamma$-law gas, the
vacuum must be avoided.  Moreover, when $1<\gamma<3$, our estimates
depend on a lower bound for the density as well as on the upper
bound.  This is also true in the isentropic case treated by Lax.

The paper is organized as follows: in Section 2, we obtain \emph{a
  priori} $L^\infty$ bounds for the density and velocity in the
one-D compressible Euler equations.  We then use this to stregthen the
gradient blowup result of~\cite{G3}.  In Section~3, we consider one-D
MHD with orthogonal magnetic field, derive Riccati type equations for
gradients, and use these to obtain a stronger singularity
formation result than that of \cite{G4}.  In Section~4, we similarly
derive Riccati type equations and present a shock formation theorem
for compressible Euler flow in a variable area duct.  This includes
the special case of three-D spherically symmetric flow in the exterior
of a ball.

\section{Compressible Euler equations}

We begin by considering the initial value problem for the compressible
Euler equations in a Lagrangian frame in one space dimension,
\begin{align}
\tau_t-u_x&=0,\label{lagrangian1}\\
u_t+p_x&=0,\label{lagrangian2}\\
(\frac{1}{2}u^2+e)_t+(u\,p)_x&=0, \label{lagrangian3}
\end{align}
where $\rho$ is the density, $\tau=\rho^{-1}$ is the specific volume,
$p$ is the pressure, $u$ is the velocity, $e$ is the internal energy,
and $x$ is the material coordinate.  The system is closed by the
Second Law of Thermodynamics,
\beq
  T\,dS = de + p\,d\tau,
\label{2TD}
\eeq
where $S$ is the entropy and $T$ the temperature.  For $C^1$ solutions,
it follows that (\ref{lagrangian3}) is equivalent to the ``entropy
equation''
\beq
   S_t=0.
\label{s con}
\eeq
When the entropy is constant, the flow is isentropic, and
(\ref{lagrangian1}) and (\ref{lagrangian2}) become a closed system,
known as the $p$-system~\cite{smoller}.

We assume the gas is ideal polytropic, so that
\[
  p\,\tau=R\,T   \com{and}  e=c_\tau\,T
\]
with ideal gas constant $R$, and specific heat $c_\tau$; this implies 
\beq
   p=K\,e^{\frac{S}{c_\tau}}\,\tau^{-\gamma}
\label{introduction 3}
\eeq
with adiabatic gas constant $\gamma = \frac R{c_\tau} +1$,
see~\cite{courant}.  The nonlinear Lagrangian sound speed is
\beq
  c:=\sqrt{-p_\tau}=
  \sqrt{K\,\gamma}\,{\tau}^{-\frac{\gamma+1}{2}}\,e^{\frac{S}{2c_\tau}}.
\label{c def}
\eeq

\subsection{Coordinates}
We use the coordinates introduced by Temple and Young in
\cite{young blake 1}. Define new variables $m$ and $\eta$ for $S$ and
$\tau$, by
\beq
   m:=e^{\frac{S}{2c_\tau}}>0,\label{m def}
\eeq
and, referring to \eq{c def},
\beq
   \eta := \int^\infty_\tau{\frac{c}{m}\,d\tau}
         = \TS\frac{2\sqrt{K\gamma}}{\gamma-1}\,
\tau^{-\frac{\gamma-1}{2}}>0.\label{z def}
\eeq
It follows that
\begin{align}
  \tau&=K_{\tau}\,\eta^{-\frac{2}{\gamma-1}},\nn\\
  p&=K_p\, m^2\, \eta^{\frac{2\gamma}{\gamma-1}},\label{tau p c}\\
  c&=c(\eta,m)=K_c\, m\, \eta^{\frac{\gamma+1}{\gamma-1}}, \nn
\end{align}
where $K_\tau$, $K_p$ and $K_c$ are positive
constants given by
\beq
  K_\tau=\TS(\frac{2\sqrt{K\gamma}}{\gamma-1})^\frac{2}{\gamma-1},
\quad
  K_p=K\,K_\tau^{-\gamma},\com{and}
  K_c=\TS\sqrt{K\gamma}\,K_\tau^{-\frac{\gamma+1}{2}},
\label{Kdefs}
\eeq
so that also
\beq
   K_p=\TS\frac{\gamma-1}{2\gamma}K_c \com{and}
   K_\tau K_c=\frac{\gamma-1}{2}.
\label{KpKcRela}
\eeq
In these coordinates, for $C^1$ solutions, equations
\eq{lagrangian1}--\eq{lagrangian3} are equivalent to
\begin{align}
  \eta_t+\frac{c}{m}\,u_x&=0, \label{lagrangian1 zm}\\
  u_t+m\,c\,\eta_x+2\frac{p}{m}\,m_x&=0,\label{lagrangian2 zm}\\
  m_t&=0\label{lagrangian3 zm},
\end{align}
the last equation being (\ref{s con}), which is equivalent to
(\ref{lagrangian3}).  Note that, while the solution remains $C^1$,
$m=m(x)$ is given by the initial data and can be regarded as
stationary.

\subsection{$L^\infty$ bounds for $C^1$ solutions}

Our first goal is to provide a uniform upper bound for the density
$\rho$ and velocity $|u|$, depending only on the initial data, for
$C^1$ solutions of (\ref{lagrangian1})--(\ref{lagrangian3}).
Denote the Riemann invariants by
\beq
  r:=u-m\,\eta,\qquad s:=u+m\,\eta.
\label{r_s_def}
\eeq
For isentropic $C^1$ solutions ($m$ constant) the Riemann invariants
$s$ and $r$ are constants along forward and backward characteristics,
respectively.  However, for general non-isentropic flow, $s$ and $r$
vary along characteristics.  

The forward and backward characteristics are described by
\[
  \frac{dx}{dt}=c \com{and} \frac{dx}{dt}=-c,
\]
and we denote the corresponding directional derivatives along these by
\[
  \pp := \dbyd t+c\,\dbyd x \com{and}
  \pn := \dbyd t-c\,\dbyd x,
\]
respectively.  Using \eq{r_s_def}, (\ref{tau p c}) and
(\ref{KpKcRela}), equations (\ref{lagrangian1 zm}) and
(\ref{lagrangian2 zm}) become
\begin{align}
  \pp s&=\frac{1}{2\gamma}\,\frac{\pp m}{m}\,(s-r),\label{s_eqn}\\
  \pn r&=\frac{1}{2\gamma}\,\frac{\pn m}{m}\,(r-s),\label{r_eqn}
\end{align}
while by (\ref{lagrangian3}), we have
\beq
  \pp m=c\, m'(x),\com{and}
  \pn m=-c\, m'(x).
\label{m_2eqn}
\eeq

Applying an integrating factor and setting
\[
  \tilde{s}=m^{-\frac{1}{2\gamma}}\,s \com{and}
  \tilde{r}=m^{-\frac{1}{2\gamma}}\,r,
\]
we obtain the system
\begin{align}
\pp\tilde{s}&=-\frac{1}{2\gamma}\,\frac{\pp m}{m}\,\tilde{r}\label{ti_s},\\
\pn\tilde{r}&=-\frac{1}{2\gamma}\,\frac{\pn m}{m}\,\tilde{s}\label{ti_r}.
\end{align}

We obtain upper bounds for $C^1$ solutions of
\eq{lagrangian1}--\eq{lagrangian3} by bounding $|\tilde{s}|$ and
$|\tilde{r}|$.  We assume that the initial entropy $S(x)$ is $C^1$ and
has finite total variation, so that
\beq
  V := \frac{1}{2c_\tau}\int_{-\infty}^{+\infty}|S'(x)|\;dx
     = \int_{-\infty}^{+\infty}\frac{|m'(x)|}{m(x)}\;dx<\infty,
\label{Vdef}
\eeq
while also, by \eq{m def},
\beq
  0 < M_L < m(\cdot) < M_U,
\label{m_bounds}
\eeq
for some constants $M_L$ and $M_U$.  Also, since $\rho$ and $|u|$ are
bounded initially, there exist positive constants $M_s$ and $M_r$,
such that, in the initial data,
\beq
  |\tilde{s}_0(\cdot)|<M_s \com{and}
  |\tilde{r}_0(\cdot)|<M_r.
\label{Mrs}
\eeq

We define two useful constants by 
\begin{align*}
  N_1 &:= M_s+\ol V\,M_r+\ol V\,(\ol V\,M_s+{\ol V}^2\,M_r)
	\,e^{{\ol V}^2},\\
  N_2 &:= M_r+\ol V\,M_s+\ol V\,(\ol V\,M_r+{\ol V}^2\,M_s)
	\,e^{{\ol V}^2},
\end{align*}
where $\ol V := \frac{V}{2\gamma}$, which clearly depend only on the
initial data.

\begin{theorem}
\label{Thm_upper}
For $C^1$ solutions of (\ref{lagrangian1})--(\ref{lagrangian3}), we
have the \emph{a priori} bounds
\[
   |u|\leq\frac{N_1+N_2}{2}{M_U}^{\frac{1}{2\gamma}}
\com{and}
   \rho\leq\frac{N_1+N_2}{2}{M_L}^{\frac{1}{2\gamma}-1},
\]
\end{theorem}

Note that the density $\rho$ need not be bounded away from zero, as
vacuums can form in infinite time~\cite{YpRP2,G1}.  The theorem is an
immediate consequence of the following estimate.

\begin{lemma}
\label{lemma1}
For a given point $(x_1,t_1)$, suppose the solution is $C^1$ in the
characteristic triangle bounded by the forward and backward
characteristics through $(x_1,t_1)$ and the line $t=0$. Then
\[
   |\tilde{s}(x_1,t_1)|\leq N_1
\com{and}
   |\tilde{r}(x_1,t_1)|\leq N_2.
\]
\end{lemma}

\begin{proof}
We prove the bound on $\tilde{s}(x_1,t_1)$; the bound of
$\tilde{r}(x_1,t_1)$ is obtained similarly.  Referring to
Figure~\ref{cyz}, we denote the forward and backward characteristics
through a point $(x_*,t_*)$ by
\begin{align*}
  \Lr_{x_*} &= \{(x,\ora t(x))\ |\ x\le x_*\} 
            = \{(\ora x(t),t)\ |\ t\le t_*\}\com{and} \\
  \Ll_{x_*} &= \{(x,\ola t(x))\ |\ x\ge x_*\} 
            = \{(\ola x(t),t)\ |\ t\le t_*\},
\end{align*}
respectively, parameterized by $x$ or $t$, as convenient.

\begin{figure}[th] \centering
  \begin{tikzpicture}[scale=0.65]
    \draw[->] (-3.5,0) -- (-3.5,4);
    \node at (-3.7,3.8) {$t$};
    \draw[->] (-3.5,0) -- (9.5,0);
    \node at (9.3,-0.3) {$x$};
    \draw (8,0) .. controls (7,2) and (6.5,3) .. (4,6);
    \draw (-2,0) .. controls (0,1.3) and (1,2) .. (4,6);
    \draw[dashed] (1.5,0) .. controls (4,1.3) and (5,2) .. (6,3.55);
    \draw[dashed] (5.4,0) .. controls (4.3,1.3) and (3.8,2) .. (2.05,3.5);
    \draw[fill] (4,6) circle (0.1);
    \draw[fill] (2.05,3.5) circle (0.07);
    \draw[fill] (5.95,3.5) circle (0.07);
    \draw[fill] (4.1,1.54) circle (0.07);
    \node at (5.2,6) {$(x_1,t_1)$};
    \node at (1,3.7) {$(x_\si,t_\si)$};
    \node at (5.4,1.5) {$(x_{\ol\si},t_{\ol\si})$};
    \node at (7,3.8) {$(x_\xi,t_\xi)$};
    \node at (-2,-0.5) {$\ora x_1(0)$};
    \node at (1.5,-0.5) {$\ora x_\xi(0)$};
    \node at (8,-0.5) {$\ola x_1(0)$};
    \node at (5.4,-0.5) {$\ola x_\si(0)$};
    \node at (0,2.2) {$\Lr_{x_1}$};
    \node at (2.3,1.15) {$\Lr_{x_\xi}$};
    \node at (7.5,2.2) {$\Ll_{x_1}$};
    \node at (3.7,2.8) {$\Ll_{x_\si}$};
  \end{tikzpicture}
  \caption{Characteristic triangle}\label{cyz}
\end{figure}
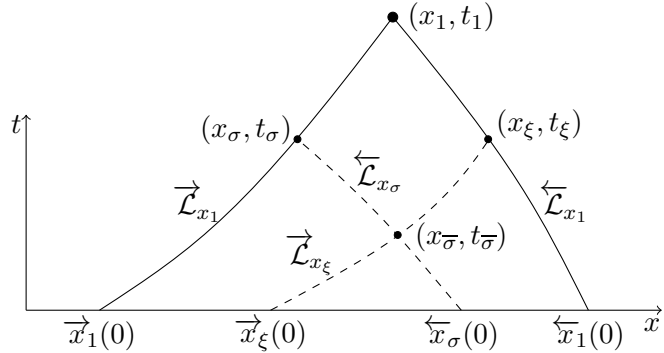

We integrate (\ref{ti_s}) along the forward characteristic $\Lr_{x_1}$
from $(\ora x_1(0),0)$ to $(x_1,t_1)$, with a change of integration
variable, and use $x_\si$ as the new parameter.  We then integrate
(\ref{ti_r}) along the backward characteristics $\Ll_{x_\si}$, from
$0$ to $t_\si$, for each $(x_\si,t_\si)\in\Lr_{x_1}$, and again change
variables.  We have
\[
  \tilde{s}(x_1,t_1) = \tilde{s}_0(\ora x_1(0)) - 
     \frac{1}{2\gamma}\,\fint_{\ora x_1(0)}^{x_1}
     \frac{m'}{m}(x_\si)\,\tilde{r}(x_\si,\ora{t}(x_\si))\;d x_\si,
\]
where we have used \eq{m_2eqn}, and similarly
\[
  \tilde{r}(x_\si,t_\si) = \tilde{r}_0(\ola{x}_\si(0)) -
     \frac{1}{2\gamma}\,\bint_{\ola{x}_\si(0)}^{x_\si}
     \frac{m'}{m}(x)\,\tilde{s}(x,\ola{t}_\si(x))\;dx.
\]
Here $\fint$ and $\bint$ indicate the direction of the characteristic
along which the integration is performed.  Combining these, we get
\begin{align}
  \tilde{s}(x_1,t_1) &= \tilde{s}_0(\ora x_1(0))
      -\frac{1}{2\gamma}\fint_{\ora x_1(0)}^{x_1}
      \frac{m'}{m}(x_\si)\,\tilde{r}_0(\ola{x}_\si(0))\;dx_\si
\label{key_estimate0}\\\nn
 &\qquad + \frac{1}{4\gamma^2}\fint^{x_1}_{\ora x_1(0)}
   \frac{m'}{m}(x_\si)\,\bigg(\bint_{\ola{x}_\si(0)}^{x_\si}
   \frac{m'}{m}(x)\,\tilde{s}(x, \ola{t}_\si(x))\,dx\bigg)\;dx_\si.
\end{align}
The first two terms can be estimated by our initial bounds, and we
will apply a Gronwall inequality to estimate the third term.

Equation \eq{key_estimate0} continues to hold for any point
$(x_\xi,t_\xi)$ on the backward characteristic $\Ll_{x_1}$, with $x_1$
replaced by $x_\xi$.  We multiply \eq{key_estimate0} by $\frac{\pp
  m}{m}(x_\xi)$, take absolute values, integrate in time along the
backward characteristic $\Ll_{x_1}$, and change to spatial variables.
Doing this, we get
\begin{align}
  \bint^{\ola{x}_1(0)}_{x_1}\big|&
     \tilde{s}\bigl(x_\xi,t_\xi)\bigr)\big|\;d\mu(x_\xi)
  \le{} \bint^{\ola{x}_1(0)}_{x_1} \big|
     \tilde{s}_0({\ora{x}_\xi(0)})\big|\;d\mu(x_\xi)  \nn\\
  &\quad{}+\frac{1}{2\gamma}\bint_{x_1}^{{\ola{x}_1(0)}}\!\!
     \fint_{{\ora{x}_\xi(0)}}^{x_\xi}
        \big|\tilde{r}_0(\ola{x}_{\ol\si}(0))\big|
        \;d\mu(x_{\ol\si})\;d\mu(x_\xi)  \label{bigint}\\\nn
  & \quad{}+\frac{1}{4\gamma^2}\bint^{{\ola{x}_1(0)}}_{x_1}\!\!
     \fint_{{\ora{x}_\xi(0)}}^{x_\xi}\!\!
     \bint^{{\ola{x}_{\ol\si}(0)}}_{x_{\ol\si}}\!\!
     \big|\tilde{s}(x, \ola{t}_{\ol\si}(x))
     \big|\;d\mu(x)\;d\mu(x_{\ol\si})\;d\mu(x_\xi),
\end{align}
where $d\mu$ is the Stieltjes measure
\[
   d\mu(x) := \bigg|\frac{m'}m(x)\bigg|\;dx,
\]
which has total mass at most $V$ along any characteristic.  Again the
first two terms of \eq{bigint} can be estimated by the initial data,
so we focus on the last term.

For $(x_\si,t_\si)$ on the forward characteristic $\Lr_{x_1}$, set
\beq
  F(x_\si) := \bint^{\ola{x}_\si(0)}_{x_\si}
     \!\!|\tilde{s}\bigl(x,\ola{t}_\si(x)\bigr)|\;d\mu(x),
\label{Fdef}
\eeq
where the integral is along the backward characteristic
$\Ll_{x_\si}$.  If the backward characteristic through $(x,
\ola{t}_{\ol\si}(x))$ is extended up to $\Lr_{x_1}$, the innermost
integral in \eq{bigint}  is estimated by
\[
  \bint^{\ola{x}_{\ol{\si}}(0)}_{x_{\ol{\si}}}
     \!\!|\tilde{s}\bigl(x,\ola{t}_{\ol{\si}}(x)\bigr)|\;d\mu(x)
 \le F(x_\si),
\]
provided $(x_{\ol{\si}},t_{\ol{\si}})\in\Ll_{x_\si}$, so that also
$\ola{x}_{\ol\si}(0) = \ola{x}_{\si}(0)$, and we change the order of
the other two integrations.  Using \eq{Fdef} with
\eq{Vdef} and \eq{Mrs}, it follows that \eq{bigint} yields
\begin{align*}
  F(x_1) &\leq V\,M_s+\frac{1}{2\gamma}V^2\,M_r+\frac{1}{4\gamma^2}
    \fint_{\ora x_1(0)}^{x_1}\!\!
    \bint^{\ola{x}_\si(0)}_{x_\si}\!\! 
      F(x_\sigma)\;d\mu(\ol x)\;d\mu(x_\si) \\
   &\leq V\,M_s+\frac{1}{2\gamma}V^2\,M_r+\frac{1}{4\gamma^2}\,V
    \fint_{\ora x_1(0)}^{x_1}\!\!F(x_\sigma)\;d\mu(x_\si).
\label{Fest}
\end{align*}

We now use Gronwall's inequality to get
\[
  F(x_1) \leq (V\,M_s+{\TS\frac{1}{2\gamma}}\,V^2\,M_r )\,
	e^{\frac{V^2}{4\gamma^2}}=:\hat{N}_1.
\]
It follows similarly that for $(x_\si,t_\si)\in\Lr_{x_1}$, we have
$F(x_\si)\le\hat N_1$.  Finally, using this estimate in
(\ref{key_estimate0}), we get
\[
  |\tilde{s}(x_1,t_1)|\leq \TS M_s+\frac{1}{2\gamma}
     V\,M_r+\frac{1}{4\gamma^2}V\,\hat{N}_1=N_1,
\]
and the proof is complete.
\end{proof}

\subsection{Singularity formation}

In \cite{G3}, we introduce gradient variables
\begin{align}
  y &= m^{-\frac{3(3-\gamma)}{2(3\gamma-1)}}\,
       \eta^{\frac{\gamma+1}{2(\gamma-1)}}\,
       ((u+m\,\eta)_x - {\TS\frac{2}{3\gamma-1}}\,m_x\,\eta)
\com{and}\nn\\
  q &= m^{-\frac{3(3-\gamma)}{2(3\gamma-1)}}\,
       \eta^{\frac{\gamma+1}{2(\gamma-1)}}\,
       ((u-m\,\eta)_x + {\TS\frac{2}{3\gamma-1}}\,m_x\,\eta),
\label{intr main}
\end{align}
and derive Riccati type equations for their evolution.

\begin{lemma}\label{them new ODEs}\cite{G3}
For $C^2$ solutions of (\ref{lagrangian1})--(\ref{lagrangian3}), we have
\begin{align} 
  \partial_+ y &= a_0+ a_2 \, y^2, \nn\\
  \partial_- q &= a_0+ a_2 \, q^2,
\label{yq odes}
\end{align}
where
\begin{align}
  {a}_0 &= {\TS\frac{K_c}{\gamma}}\,
        \big[{\TS\frac{\gamma-1}{3\gamma-1}}\,m\,m_{xx}
         - {\TS\frac{(3\gamma+1) (\gamma-1)}{(3\gamma-1)^2}}\,m_x^2\big]\,
	m^{-\frac{3(3-\gamma)}{2(3\gamma-1)}}\,
         \eta^{\frac{3(\gamma+1)}{2(\gamma-1)}+1},\nn\\
  {a}_2 &= -K_c\,{\TS\frac{\gamma+1}{2(\gamma-1)}}\,
	m^{\frac{3(3-\gamma)}{2(3\gamma-1)}}\,
        \eta^{\frac{\gamma+1}{2(\gamma-1)}-1}<0.
\label{adefs}
\end{align}
Furthermore,
\beq
  |y|\ \it{or}\ |q| \to \infty \
  \com{iff}\ 
  |u_x|\ \it{or}\ |\tau_x| \to \infty.
\label{y q blowup}
\eeq
\end{lemma}

Similar equations are derived independently in~\cite{linliuyang}.
Using these equations, a singularity formation result was proved in
\cite{G3}.  By applying Lemma \ref{lemma1}, we can improve that result
to obtain our main theorem on the Euler equations.

\begin{theorem}
\label{Thm singularity2}
Assume the initial data are $C^2$ with entropy having bounded
variation, and suppose there is a positive constant $M_*$ such that
the initial entropy satisfies $|m''(x)|<M_*$.  If $1<\gamma<3$, we
also assume that the density has a positive global lower bound.  There
exists positive constant $N$ depending only on the initial data, such
that, if the initial data satisfies
\beq
  \inf\;\{\  y(\cdot,0),\ q(\cdot,0)\ \} < -N,
\label{yq-N}
\eeq
then  $|u_x|$ and/or $|\tau_x|$ blow up in finite time.
\end{theorem}

This result is an extension of Lax's singularity formation result in
\cite{lax2} for $2\times2$ strictly hyperbolic systems.  We recover
the isentropic case by taking $m$ constant, so $a_0=0$.  The
conditions which guarantee blowup are expressed in terms of global
bounds for the variables, so we must ensure that the state does not
leave some compact set, and, in particular, the vacuum must be
avoided.  When $1<\gamma<3$, our estimates depend on a lower bound for
the density as well as on the upper bound.  This is also true in the
isentropic case treated by Lax.  However, when $1<\gamma<3$, we expect
a more refined analysis to show that even if the solution approaches
vacuum, the gradient still blows up in finite time~\cite{Ygdbu}.

\begin{proof}
By Theorem \ref{Thm_upper} and (\ref{z def}), we know $\eta$ has a
global upper bound depending only on the initial data, denoted $E_U$.

By (\ref{adefs}), it is easy to calculate that, if $a_0\geq0$,
\[
  \sqrt{-\frac{a_0}{a_2}} =
  \sqrt{{\TS\frac{2(\gamma-1)^2}{\gamma(\gamma+1)(3\gamma-1)}}\,
        \big(m\,m_{xx}-{\TS\frac{3\gamma+1}{3\gamma-1}}\,m_x^2\big)}\,
    \eta^{\frac{\gamma+1}{2(\gamma-1)}+1}\,m^{-\frac{3(3-\gamma)}{2(3\gamma-1)}},
\]
which implies the uniform bound $\sqrt{-{a_0}/{a_2}} \le N$, where
\beq
  N :=
  \begin{cases}
  \sqrt{\frac{2(\gamma-1)^2}{\gamma(\gamma+1)(3\gamma-1)}\,M_*}
        \  E_U^{\frac{3\gamma-1}{2(\gamma-1)}}\  
        M_L^{\frac{3\gamma-5}{3\gamma-1}}, & 1<\gamma\le 5/3,\\
  \sqrt{\frac{2(\gamma-1)^2}{\gamma(\gamma+1)(3\gamma-1)}\,M_*}
        \,E_U^{\frac{3\gamma-1}{2(\gamma-1)}}\,
        M_U^{\frac{3\gamma-5}{3\gamma-1}}, & \gamma\ge 5/3.
  \end{cases}
\label{Ndef}
\eeq     
It follows that if $y<-N$, then $a_0+a_2\,y^2<0$.

Now suppose that \eq{yq-N} holds.  Then there exist $\eps>0$ and $x_0$
such that
\[
  y(x_0,0) < -(1+\eps)\,N,
\]
say.  Now consider the forward characteristic starting at $(x_0,0)$.
By \eq{yq odes}, along this characteristic (parametrized by $t$) we
have $\pp y < 0$, so also
\[
   y(t) < -(1+\eps)\,N \com{for} t\ge 0,
\]
which in turn implies
\[
   a_0 + a_2\,\frac{y(t)^2}{(1+\eps)^2} < 0
\]
for all $t \ge 0$.  Now (\ref{yq odes}) implies
\[ 
  \pp y = a_0 + a_2 \,y^2
       < (1-{\TS\frac{1}{(1+\eps)^2}})\, a_2\,y^2<0,
\]
since $a_2<0$.  Integrating, we get 
\beq
  \frac{1}{y(t)} \ge {\frac{1}{y(0)} -
    \fint_0^t (1-{\TS\frac{1}{(1+\varepsilon)^2}})\, {a_2}\;dt},
\label{SS9 1}
\eeq 
where the integral is along the forward characteristic.  By
(\ref{adefs}), when $\gamma\geq 3$, $a_2$ is negative and bounded
above, so the right hand side of (\ref{SS9 1}) approaches zero in
finite time.  This implies that $y(t)$ approaches $-\infty$ in finite
time, so that $|\tau_x|$ and/or $|u_x|$ blow up.

When  $1<\gamma<3$, and assuming $\eta$ has a positive (global) lower
bound  along the characteristic, then $a_2$ is again negative and
bounded above, so the result follows similarly.  It is clear that the
same argument holds along a backward characteristic if $\inf q<-N$.
\end{proof}

\section{Magnetohydrodynamics}

\subsection{Prior results for generalized Euler equations}
\label{section_mhd}

We recall the results of \cite{G4} for the generalized $p$-system
(which includes the smooth Euler equations) in one space dimension,
\begin{align}
  \tau _t-u_x &= 0,\nn\\
  u_t+P(\tau ,x)_x &= 0,
\label{gpsys}
\end{align} 
where $P(\tau,x)$ is a $C^3$ function of $\tau>0$ and $x$, satisfying
\[
  P_{\tau }<0 \com{and} P_{\tau\tau}>0,
\]
so that the system is hyperbolic with wavespeed
\[
  c := c(\tau,x) = \sqrt{-P_\tau}.
\]
It is convenient to introduce new variables $(h,\mu)$ for $(\tau ,x)$,
by setting 
\[
  h(\tau ,x) :=\int_\tau^{\tau^*}c\; d\tau
     = \int_\tau ^{\tau ^{*}}\,\sqrt{-P_\tau}\;d\tau \com{and} \mu=x,
\]
where $\tau ^*$ is a constant or infinity.
It follows that
\beq
  h_\tau=-c, \quad
  \tau _h=-\frac{1}{c},  \com{and}
  P_h = c,
\label{mhd_htp}
\eeq
and furthermore, for any $C^1$ function $f(\tau,x)$,
\beq
  f_h = -\frac{f_\tau }{c} \com{and}
  f_\mu = \frac{h_x}{c}\,f_\tau + f_x.
\label{mhd_f}
\eeq

We define gradient variables by
\begin{align}
  y := \sqrt{c}\,(u+h)_x + \frac{P_\mu}{\sqrt{c}}-I, 
\nn\\\label{mhd_yq}
  q := \sqrt{c}\,(u-h)_x - \frac{P_\mu}{\sqrt{c}}+I,
\end{align}
where
\beq
  I = I(h,\mu) =
  \int_{h_0}^{h} {\frac{1}{2}\,\sqrt{c}
	\,\left(\frac{P_\mu}{c}\right)_h}\;dh,
\label{Idef}
\eeq
and $h_0$ is a constant.

\begin{lemma}\cite{G4} 
\label{mhd them new ODEs}
For $C^2$ solutions of (\ref{gpsys}), we have
\begin{align}
  \pp y &= a_0+a_1\, y + a_2\,y^2, \nn\\
  \pn q &= a_0-a_1\, q + a_2\, q^2,
\label{mhdyqde}
\end{align} 
where 
\begin{align}
  a_0 &= -c\,I_\mu + 
    \frac{1}{2}\,\sqrt{c}\,\left(\frac{P_\mu}{c}\right)_h\,P_\mu -
    c\,\left(\frac{P_\mu}{c}\right)_h\, I-\frac{c_h}{2\sqrt{c}}\,I^2,
\nn\\\label{mhd_a}
  a_1 &= -\big(2\,\sqrt{c}\,I\big)_h =
    -c\,\left(\frac{P_\mu}{c}\right)_h-\frac{c_h}{\sqrt{c}}\,I,\\\nn
  a_2 &= -\frac{c_h}{2\sqrt{c}} < 0.
\end{align}
\end{lemma}

For arbitrarily given positive constants $A_i$ and $B_i$, we denote by
$\mathcal K$ the compact set whose interior $\mathcal{K}^o$ is given by
\begin{gather*}
  |h|<B_1, \quad A_2<c<B_2, \quad 
  A_3<c_h<B_3, \quad  |c_\mu|<B_4 \\
  |c_{\mu\mu}|<B_5, \quad |c_{h\mu}|<B_6, \quad 
  |p_\mu|<B_7, \quad |p_{\mu\mu}|<B_8.
\end{gather*}

\begin{lemma}\cite{G4}
\label{mhd Thm singularity2}
There exists a constant $\tilde{N}>0$ depending only on $\mathcal K$, such
that, if the $C^2$ initial data of (\ref{gpsys}) satisfy
\[
  (u(x,0),\;\tau(x,0)) \in \mathcal{K}^o,\com{for all} x,
\]
and if
\[
  \inf\;\{\ y(\cdot,0),\ q(\cdot,0)\ \} < -\tilde{N},
\]
then there exists $T_*=T_*(\mathcal K,\tilde{N})$ such that \emph{either}
\[
   \max\left\{|u_x|,|\tau_x|\right\} \to \infty \com{as} t\to T_*,
\]
\emph{or} there is some point $(x_b,t_b)$ with $t_b \le T_*$ such that
\[
  (u(x_b,t_b),\;\tau(x_b,t_b)) \in \partial\mathcal{K}.
\]
\end{lemma}

\subsection{Refinements for MHD}

In this section, we apply and upgrade Lemma~\ref{mhd Thm
  singularity2} for one-D Magnetohydrodynamics.  The system
models the motion of a smooth compressible fluid coupled to a magnetic
field $H=(H_1, H_2, H_3)$.  In a Lagrangian frame in one space
dimension, $C^1$ solutions satisfy the quasilinear system
\begin{align}
  \frac{\partial\tau}{\partial t} - 
    \frac{\partial u_1}{\partial x} &= 0, \nn\\
  \frac{\partial H_j}{\partial t} +
    \rho \,H_j \,\frac{\partial u_1}{\partial x} -
    \rho \,H_1 \,\frac{\partial u_j}{\partial x} &= 0,
\quad j = 2,3 \nn\\
  \frac{\partial u_1}{\partial t} +
    \frac{\partial }{\partial x}\,
    \big(p+{\TS\frac{1}{2}}\,\mu_0\,(H^2_2+H^2_3)\big) &= 0,
\label{fullmhd}\\
  \frac{\partial u_j}{\partial t} -
    \mu_0 \,H_1\,\frac{\partial H_j}{\partial x} &= 0,
\quad j = 2,3  \nn\\
  \frac{\partial S}{\partial t} &= 0,\nn
\end{align}
Here the fluid quantities are the density $\rho$, specific volume
$\tau=\rho^{-1}$, velocity field $(u_1,u_2,u_3)$, entropy $S$ and
hydrostatic pressure $p=p(\tau,S)$, while $(H_1,H_2,H_3)$ is the
magnetic field and $\mu_0$ is the magnetic constant~\cite{Li Qin}.  In
one-dimensional MHD, the first component $H_1$ of the magnetic field
is necessarily constant.

System (\ref{fullmhd}) simplifies significantly if we take the
constant $H_1$ to vanish, so the magnetic field is perpendicular to
the direction of motion.  In this case, the full system \eq{fullmhd}
can be written
\begin{align}
  \frac{\partial \tau }{\partial t} - 
     \frac{\partial u_1}{\partial x} &= 0,  \nn\\
  \frac{\partial u_1}{\partial t} +
     \frac{\partial \tilde{p}}{\partial x} &= 0,
\label{mhd}
\end{align}
regarded as a generalized $p$-system, and coupled with
\beq
  \frac{\partial \tilde{H}_j}{\partial t} = 0, \quad
  \frac{\partial {u_j}}{\partial t} = 0, \quad
  \frac{\partial S}{\partial t} = 0, \quad j = 1,2,
\label{mhdstat}
\eeq 
where we have written
\begin{gather*}
  \tilde H_j := \tau\, H_j   \com{and}
  \tilde{p} := p+\frac{1}{2}\,\mu_0\,
     \frac{{\tilde{H}}^2_2+{\tilde{H}}^2_3}{\tau^2}.
\end{gather*}

System (\ref{mhd}) is clearly an example of (\ref{gpsys}), so all
results in \cite{G4} apply directly.  Here we obtain a stronger
singularity formation result than Lemma \ref{mhd Thm singularity2} by
restricting our consideration to a polytropic ideal gas,
\[
  p = K\,e^{\frac{S}{c_\tau }}\,\tau^{-\gamma}, 
\]
with adiabatic gas constant $1<\gamma\le2$.

\begin{theorem}
\label{MHD_sing}
For system (\ref{mhd}), \eq{mhdstat}, assume that the initial data is
$C^2$ with uniform bounds
\[
  |G(\cdot,0)| < N_0, \quad 
  |G(\cdot,0)_x| < N_1, \com{and}
  |G(\cdot,0)_{xx}| < N_2
\]
for each of $G = S$, $H_2$, $H_3$.  Assume also that the
density is globally bounded,
\beq
  0 < N_{\rho L} < \rho(x,t) < N_{\rho U}, \com{for any}
  (x,t)\in\RR\times\RR^+,
\label{taubd}
\eeq
for constants $N_{\rho L}$ and $N_{\rho U}$.  There exists a positive
constant $\hatN$, depending on $N_j$ and $N_{\rho U}$, but not on
$N_{\rho L}$, such that, if
\[
   \inf\;\{\ y(\cdot,0),\  q(\cdot,0)\ \}<-\hatN,
\] 
then $u_x$ and/or $\tau_x$ blow up in finite time.  Here $y$ and $q$
are defined in (\ref{mhd_yq}), with $P=\tilde{p}$.
\end{theorem}

Although we assume the global bounds \eq{taubd}, $\hatN$ is
independent of the lower bound $N_{\rho L}$ of density.  Theorem
\ref{Thm_upper} for the compressible Euler equations indicates that it
is reasonable to expect an \emph{a priori} upper bound on the density,
and such a bound would imply that $\hatN$ depends only on the initial
data.

\begin{proof} 
As in the proof of Theorem \ref{Thm singularity2}, from \eq{mhdyqde},
it suffices to find uniform bounds for the roots of the quadratic
equations
\[
  a_0 + a_1\,y + a_2\,y^2 = 0 \com{and}
  a_0 - a_1\,q + a_2\,q^2 = 0,
\]
whose coefficients are given in (\ref{mhd_a}) with $P$ replaced by
$\tilde{p}$.  That is, we must choose $\hatN$ so that
\beq
  \left|\frac{a_1}{2\,a_2}\right|
  + \sqrt{\left|\frac{a_1}{2\,a_2}\right|^2 + \frac{a_0}{|a_2|}}
  \le \hatN,
\label{hatN}
\eeq
when $a_1^2-4 a_0 a_2\geq 0$, where we calculate
\[
   \left|\frac{a_1}{2\,a_2}\right| =
   \left|\frac{c\sqrt{c}}{c_h}\,
        \bigg(\frac{\tilde{p}_\mu}{c}\bigg)_{\!\!h} + I\right|,
\]
and 
\[
  \left|\frac{a_1}{2\,a_2}\right|^2 + \frac{a_0}{|a_2|}  =
  \left(\frac{c\sqrt{c}}{c_h}\,
        \bigg(\frac{\tilde{p}_\mu}{c}\bigg)_{\!\!h\;}\right)^{\!\!2}
      - 2\,\frac{c\sqrt{c}}{c_h}\,I_\mu
      + \frac{c}{c_h}\,
      \bigg(\frac{\tilde{p}_\mu}{c}\bigg)_{\!\!h}\,\tilde{p}_\mu.
\] 
It follows that we can choose $\hatN$ provided we find upper bounds
(independent of $N_{\rho L}$) for each of the quantities
\beq
  \left|\bigg(\frac{\tilde{p}_\mu }{c}\bigg)_{\!\!h}\right|, \quad
  \left|\frac{\tilde{p}_\mu}{\sqrt{c}}\right|,  \quad
  \left|\frac{c\sqrt{c}}{c_h}\right|, \quad
  |I|, \com{and} |I_\mu|;
\label{req}
\eeq
we treat each of these terms separately.

In order to simplify the calculation, we denote 
\[
  A(x) := K\,e^{\frac{S}{c_\tau }},\quad
  B(x) := \frac{1}{2}\,\mu_0\,(\tilde{H}^2_2+{\tilde{H}}^2_3),
\]
so that $A(x)$ and $B(x)$ are positive bounded functions, and $A(x)$
is bounded from below by a positive constant.  With this notation,
\[
  \tilde{p}(\tau,x)=A(x)\,\tau^{-\gamma} + B(x)\,\tau^{-2},
\]
the wave speed is
\beq
  c = \sqrt{-\tilde{p}_\tau }
    = \sqrt{\gamma\,A\,\tau^{-\gamma-1} + 2\,B\,\tau^{-3}}, 
\label{mhd_c_def}
\eeq 
and 
\[
  h = \int_\tau^\infty c\;d\tau
    = \int_\tau^\infty
      \sqrt{\gamma\,A\,\tau^{-\gamma-1} + 2\,B\,\tau^{-3}}\;d\tau.
\]
It follows by \eq{taubd} that $\tilde p$, $c$ and $h$ are bounded
above and below.  We write 
\[
  0 < N_{hL} < h < N_{hU},
\]
where $N_{hU}$ does not depend on $N_{\rho L}$.

{\bf 1.} We first consider $(\frac{\tilde{p}_\mu }{c})_h$. By
(\ref{mhd_htp}) we have $\tilde{p}_{\mu h}=c_\mu$, and by \eq{mhd_f},
\[
  c\,\left(\frac{\tilde{p}_\mu}{c}\right)_h
  = c_\mu - c_h\,\frac{\tilde{p}_\mu}{c}
  = c_{x} + \frac{c_\tau}{c^2}\,\tilde{p}_{x}
  = \frac{c}{2}\,\left(\frac{\tilde{p}_{x}}{\tilde{p}_\tau }\right)_\tau,
\]
and thus
\[
  \left(\frac{\tilde{p}_\mu}{c}\right)_h =
  \frac{1}{2}\,\left(\frac{\tilde{p}_{x}}{\tilde{p}_\tau}\right)_\tau
  = \frac{\tilde{p}_{x\tau}\,\tilde{p}_\tau 
          - \tilde{p}_x\,\tilde{p}_{\tau\tau}}{2\tilde{p}_\tau^2}.
\]
Now, since $1<\gamma\le 2$, we have
\[
  |\tilde p_\tau| 
  = \gamma\,A\,\tau^{-\gamma-1} + 2\,B\,\tau^{-3}
  \ge \gamma\,A\,\tau^{-\gamma-1},
\]
and since $A$ is bounded from below, $(\frac{\tilde{p}_\mu}{c})_h$ is
bounded provided
\[
  F_1 = (\tilde{p}_{x\tau}\,\tilde{p}_\tau 
          - \tilde{p}_x\,\tilde{p}_{\tau\tau})\,\tau^{2\gamma+2}
\]
is bounded.  It is routine to compute
\[
  F_1 = (\gamma A_x + 2B_x\tau^{\gamma-2})
        (\gamma A + 2B\tau^{\gamma-2})
        - (A_x + B_x\tau^{\gamma-2})
        (\gamma(\gamma+1)A + 6B\tau^{\gamma-2}),
\]
which is bounded independent of $N_{\rho L}$ for $\gamma \le 2$.

{\bf 2.}  By (\ref{mhd_f}),  
\[
  \frac{\tilde{p}_\mu}{\sqrt{c}}
  = \frac{\tilde{p}_x}{\sqrt{c}}
  + h_x\,\frac{\tilde p_\tau}{c\,\sqrt c}
  = \frac{\tilde{p}_x}{\sqrt{c}}
  - \sqrt{c}\,h_x,
\]
and this is bounded by a constant independent of $N_{\rho L}$, since
\[
  \frac{\tilde{p}_{x}}{\sqrt{c}} =
  \frac{A_x\,\tau^{-\gamma}+B_x\,\tau^{-2}}
    {(\gamma\,A\,\tau^{-\gamma-1}+2\,B\,\tau^{-3})^{\frac{1}{4}}}
  = O(1)\,\tau^{-\frac{3\gamma-1}4},
\]
and
\[
  h_x =\int_\tau^{\infty}
   \frac{\gamma A_x\,\tau^{-\frac{\gamma+1}{2}}
         + 2 B_x\,\tau^{\frac{\gamma-5}{2}}}
   {2\sqrt{\gamma A+2B\,\tau^{\gamma-2}}}\;d\tau 
   = O(1)\,\tau^{-\frac{\gamma-1}2},
\]
so $\sqrt c\,h_x = O(1)\,\tau^{-\frac{3\gamma-1}4}$, where $O(1)$ is
a bound depending on $x$ and the lower bound for $\tau$, and thus
is independent of $N_{\rho L}$.

{\bf 3.}  We similarly calculate
\[
  \frac{c\sqrt{c}}{c_h}
  = \frac{c\sqrt{c}}{c_\tau\,\tau_h}
  = -\frac{2\,c^\frac{7}{2}}{2\,c\,c_\tau}
  = \frac{2\,(\gamma A\tau^{-\gamma-1}+2B\tau^{-3})^\frac{7}{4}}
    {\gamma(\gamma+1)A\tau^{-\gamma-2} + 6B\tau^{-4}}
  = O(1)\,\tau^{-\frac{3\gamma-1}4},
\]
which is again bounded independent of $N_{\rho L}$.

{\bf 4.}  Since the upper bounds of $c$, $h$ and
$|(\frac{\tilde{p}_\mu}{c})_h|$ don't depend on $N_{\rho L}$, the
absolute value of $I$ is bounded above by a constant independent of
$N_{\rho L}$.

{\bf 5.} By \eq{Idef}, (\ref{mhd_f}), 
\[
   I_\mu = \int_{h_0}^h
      \left(\frac{1}{2}\,\sqrt{c}\,
      \left(\frac{\tilde{p}_\mu}{c}\right)_{\!\!h\;}\right)_{\!\!\mu}
    \;dh = \int_{h_0}^h
      \left[\frac{c_\mu}{4\sqrt{c}}
        \left(\frac{\tilde{p}_\mu}{c}\right)_{\!\!h\;} +
        \frac{\sqrt{c}}{2}\,\left(\frac{\tilde{p}_\mu }{c}
        \right)_{\!\!h\mu\;}\right]\;dh,
\]
so it suffices to bound $\frac{c_\mu}{\sqrt{c}}$ and
$\sqrt{c}(\frac{\tilde{p}_\mu }{c})_{h\mu }$. 
By (\ref{mhd_f}), 
\[ 
  \left|\frac{c_\mu}{\sqrt{c}}\right| \le
    \left|\frac{(c^2)_\tau\,h_x}{2c^2\sqrt{c}}\right|
    + \left|\frac{(c^2)_x}{2c\sqrt{c}}\right| \le
    O(1)\,\tau^{-\frac{\gamma+1}4},
\]
as above.  Finally, estimating as above, we obtain
\[
  \sqrt{c}\,\left(\frac{\tilde{p}_\mu}{c}\right)_{h\mu } =
  O(1)\,\tau^{-\frac{\gamma+1}4},
\] 
as required.

Since all terms in \eq{req} are uniformly bounded, it follows that we
can choose $\hatN$ so that \eq{hatN} holds.  Moreover, $\hatN$ is
independent of $N_{\rho L}$.  The remainder of the proof follows
that of Theorem~\ref{Thm singularity2}.
\end{proof}

\section{Compressible flow in a variable area duct}
\subsection{Equations and coordinates}

In this section, we first consider inviscid compressible flow in a
duct of varying cross section $a(x')$.  In Eulerian (spatial)
coordinates $(x',t')$, this flow is described by
\begin{align}
  a_{t'} &= 0,  \nn\\
  (a\,\rho)_{t'} + (a\,\rho\,u)_{x'} &= 0,  \nn\\
  (a\,\rho\,u)_{t'} + (a\,\rho\,u^2)_{x'} + a\,p_{x'} &= 0, 
\label{duct}\\\nn
  (a\,\rho\,E)_{t'} + (a\,\rho\,E\,u + a\,p\,u)_{x'} &= 0,
\end{align}
$a=a(x')$ is the cross-sectional area of the duct and
$E=e+\frac12\,u^2$ is the total energy~\cite{courant,Dafermos,Hong
  Temple}.  The others are the standard thermodynamic variables,
subject to \eq{2TD} as usual.  The first equation describes the fixed
duct, while the others represent conservation of mass, momentum and
energy, respectively.  For $C^1$ solutions, the energy equation is
equivalent to an entropy equation, 
\[
  S_{t'} + u\,S_{x'} = 0.
\]

We use Lagrangian (material) coordinates $(x,t)$, defined by
\[
  x = \int a(x')\,\rho(x')\;dx', \quad t = t',
\]
so that
\beq
  dx = a\,\rho\;dx' - a\,\rho\,u\;dt',  \com{and}
  dt = dt'.
\label{ductlag}
\eeq
We define a specific length by
\beq\label{duct v1}
  v := \frac{1}{a\rho},
\eeq 
and rewrite system \eq{duct} in the Lagrangian frame, to get
\begin{align}
  v_t - u_x &= 0,   \nn\\
  u_t + a\,p_x &= 0,   \label{lduct}\\
  S_t &= 0,    \nn
\end{align}
valid for $C^1$ solutions, which is a variant of the generalized
$p$-system \eq{gpsys}.

We restrict our attention to a polytropic ideal gas (\ref{introduction
  3}), so that
\[
  p = p(v,a,S) = K\,e^{\frac S{c_\tau}}\,a^{-\gamma}\,v^{-\gamma},
\]
and define the generalized Lagrangian wavespeed by
\[
  C = \sqrt{-a\,p_{v}}
    = \sqrt{K\gamma}\,a^{-\frac{\gamma-1}{2}}\,
    v^{-\frac{\gamma+1}{2}}\,e^ {\frac{S}{2c_\tau}}.
\]

As above, we define new variables $m$ and $z$ for $S$ and $v$ by
\[
   m = e^{\frac{S}{2c_\tau}}  \com{and}
   z = \int^\infty_{v}\frac{C\,a^{\frac{\gamma-1}{2}}}{m}\;dv =
       \frac{2\sqrt{K\gamma}}{\gamma-1}\,v^{-\frac{\gamma-1}{2}}.
\]
It follows that
\begin{gather}
  v = K_{\tau}\,z^{-\frac{2}{\gamma-1}}, \qquad
  p = K_p\,a^{-\gamma}\,m^2\,z^{\frac{2\gamma}{\gamma-1}},\nn\\
\com{and}
  C = C(m,z,a) = K_c\,
          a^{-\frac{\gamma-1}{2}}\,m\,z^{\frac{\gamma+1}{\gamma-1}},
\label{vpC}
\end{gather}
where the constants $K$ are given by (\ref{Kdefs}).  In these
coordinates, $C^1$ solutions of \eq{lduct} satisfy the equivalent
system
\begin{align}
  z_t + \frac{C}{m}\,a^{\frac{\gamma-1}{2}}\,u_x &= 0,  \nn\\
  u_t + m\,C\,a^{-\frac{\gamma-1}{2}}\,z_x
      + 2\,\frac{a\,p}{m}\,m_x - \gamma\,p\,a_x &= 0,
\label{duct zm}\\\nn
  m_t &= 0.
\end{align}

\subsection{Coupled Riccati equations}

Following~\cite{G3}, we define gradient variables by
\begin{align}
  \alpha &= u_x + a^{-\frac{\gamma-1}{2}}\,m\,z_x
     + {\TS\frac{\gamma-1}{\gamma}}\,a^{-\frac{\gamma-1}{2}}\,m_x\,z,
\nn\\\label{duct ab}
   \beta &= u_x - a^{-\frac{\gamma-1}{2}}\,m\,z_x
     - {\TS\frac{\gamma-1}{\gamma}}\,a^{-\frac{\gamma-1}{2}}\,m_x\,z,
\end{align}
and denote derivatives along forward and backward characteristics by
\[
   \partial_+ = \partial_t + C\,\partial_x \com{and}
   \partial_- = \partial_t - C\,\partial_x,
\]
respectively.

Although the area $a(x')$ of the duct is stationary in spatial
coordinates, it is moving in material coordinates.  We write
\[
  \dot{a} := \frac{d a(x')}{dx'}  \com{and}
  \ddot{a} := \frac{d^2 a(x')}{d{x'}^2},
\]
so that, according to (\ref{ductlag}), $a(x,t)$ satisfies
\beq
  a_t = u(x,t)\,\dot{a}(x'),  \quad
  a_x = v(x,t)\,\dot{a}(x')  \com{and}
  a_{xx} = v^2\,\ddot a + v_x\,\dot a,
\label{adot}
\eeq
for $x'=x'(x,t)$.

\begin{lemma}
For $C^1$ solutions of \eq{duct}, we have
\begin{align}
  \pp\alpha &= k_1\,\bigl(k_2\,(3\alpha+\beta)
     + (\alpha\,\beta-\alpha^2)\bigr)
     + k_3^+\,(\alpha-\beta) + F(x,t),
\nn\\\label{duct alpha}
  \pn\beta &= k_1\,\bigl(-k_2\,(\alpha+3\beta)
     + (\alpha\,\beta-\beta^2)\bigr)
     + k_3^-\,(\beta-\alpha) + F(x,t),
\end{align}
where
\begin{gather*}
  k_1 = {\TS\frac{\gamma+1}{2(\gamma-1)}}
     \,K_c\,z^{\frac{2}{\gamma-1}},
\qquad
  k_2 = {\TS\frac{\gamma-1}{\gamma(\gamma+1)}}
     \,m_x\,z\,a^{-\frac{\gamma-1}{2}},\\
  k_3^\pm = - {\TS\frac{\gamma-1}{4}}\,u\,a^{-1}\dot{a}
    \pm{\TS\frac{3(\gamma-1)^2}{8}}
     \,m\,z\,a^{-\frac{\gamma+1}{2}}\,\dot{a},
\end{gather*}
and where
\[
  F(x,t) = {\TS\frac{(\gamma-1)^3}{8K_c}}\,m^2\,
       z^{\frac{2\gamma-4}{\gamma-1}}\,a^{-\gamma-1}\,
      (a\,\ddot{a}-\gamma\,\dot{a}^2)+
     {\TS\frac{(\gamma-1)^2}{2\gamma}}
      \,m\,m_x\,z^2\,a^{-\gamma}\,\dot{a}.
\]
\end{lemma}

We note that the coefficients $k_j$ and zero-order terms $F$ depend
only on the solution itself and on stationary derivatives of initial
entropy and duct shape.  It is reasonable to assume that $|\dot{a}|$,
$|\ddot{a}|$ and $|m_x|$ are uniformly bounded, so the coefficients
and inhomogeneity are essentially zero order terms.  As in all of the
simpler examples, the quadratic nonlinearity drives the growth of
$\alpha$ and/or $\beta$ to infinity, at which time derivatives blow up
and shocks form in the solution.

\begin{proof}
We derive the equation for $\alpha$.  Differentiating (\ref{duct ab}),
we get
\begin{align*}
  \pp\alpha = {}&
    \big[u_{xt} + C\,(a^{-\frac{\gamma-1}{2}}\,m\,z_x)_x\big] +
    \big[(a^{-\frac{\gamma-1}{2}}\,m\,z_x)_t + C\,u_{xx}\big]\\
&\qquad{}+{\TS\frac{\gamma-1}{\gamma}}
            \,\big[(a^{-\frac{\gamma-1}{2}}\,m_x\,z)_t +
            C\,(a^{-\frac{\gamma-1}{2}}\,m_x\,z)_x\big].
\end{align*}
Now differentiate (\ref{duct zm}) in $x$ to get
\begin{align*}
  (a^{-\frac{\gamma-1}{2}}\,m\, z_x)_t + C\,u_{xx} 
     &= - C_x\,u_x - (a^{-\frac{\gamma-1}{2}}\,m)_x\,z_t
        + ( a^{-\frac{\gamma-1}{2}} )_t\,m\,z_x,  \\
  u_{xt} + C\,(m\,a^{-\frac{\gamma-1}{2}}\,z_x)_x
     &= - ( 2\,\frac{a\,p}{m}\,m_x)_x
      +(\gamma\,p\,a_x)_x - C_x\,m\,a^{-\frac{\gamma-1}{2}}z_x,
\end{align*}
and substitute in to get
\begin{align}
  \pp\alpha &= - ( 2\,\frac{a\,p}{m}\,m_x)_x
      +(\gamma\,p\,a_x)_x - C_x\,m\,a^{-\frac{\gamma-1}{2}}z_x  \nn\\
     &\qquad{} - C_x\,u_x - (a^{-\frac{\gamma-1}{2}}\,m)_x\,z_t
        + ( a^{-\frac{\gamma-1}{2}} )_t\,m\,z_x  \label{dalpha}\\\nn
     &\qquad{} + {\TS\frac{\gamma-1}{\gamma}}
            \,\big[(a^{-\frac{\gamma-1}{2}}\,m_x\,z)_t +
            C\,(a^{-\frac{\gamma-1}{2}}\,m_x\,z)_x\big].
\end{align}
Note that we have eliminated all second derivatives except for
$a_{xx}$ and $m_{xx}$, since $m_{xt}=0$, and  we use \eq{adot} to
eliminate $a_t$, $a_x$ and $a_{xx}$.  The quadratic terms in \eq{duct
  alpha} are those involving $C_x$.

Finally, we express the derivatives $u_x$, $z_t$ and $z_x$ in terms of
$\alpha$ and $\beta$.  By (\ref{duct ab}), (\ref{duct zm}) and
(\ref{vpC}), we have
\begin{align*}
  u_x &= {\TS\frac{1}{2}}\,(\alpha+\beta), \\
  z_x &= {\TS\frac{a^{\frac{\gamma-1}{2} }}{2\,m}}\,(\alpha-\beta)
       - {\TS\frac{\gamma-1}{\gamma\,m}}\,z\,m_x, \\
  z_t &= - {\TS\frac{K_c}{2}}
       \,z^{\frac{\gamma+1}{\gamma-1}}\,(\alpha+\beta).
\end{align*}
Substituting these into \eq{dalpha} and simplifying yields \eq{duct
  alpha}.  We omit the details, which are similar to those appearing
in \cite{G3}.
\end{proof}

\subsection{Decoupled Riccati equations}

As in \cite{G3,G4}, equations \eq{duct alpha} can be rewritten as a
decoupled system with varying coefficients.  Since the $\alpha$
equation has terms linear in $\beta$, these can be eliminated by an
integrating factor.  By (\ref{vpC}), (\ref{duct zm}) and (\ref{duct
  ab}),
\[
  \pp z = z_t + C\,z_x = 
      - K_c\,z^{\frac{\gamma+1}{\gamma-1}}\,(\beta +
      {\TS\frac{\gamma-1}{\gamma}}\,m_x\,z\,a^{-\frac{\gamma-1}{2}}),
\]
so that
\[
  \beta = -\frac{1}{K_c}\,z^{-\frac{\gamma+1}{\gamma-1}}\,\pp z
  - {\TS\frac{\gamma-1}{\gamma}}\,m_x\,z\,a^{-\frac{\gamma-1}{2}}.
\]
We substitute for $\beta$ in (\ref{duct alpha}), move
terms including $\pp z$ to the left hand side, and multiply
by $z^{\frac{\gamma+1}{2(\gamma-1)}}$, to get
\begin{align}
\pp&(z^{\frac{\gamma+1}{2(\gamma-1)}}\,\alpha) +
{\TS\frac{1}{2\gamma}}m_x a^{-\frac{\gamma-1}{2}}z^{\frac{\gamma+1}
{2(\gamma-1)}}\partial_+ z-{\TS\frac{k_3^+}
{K_c}} z^{-\frac{\gamma+1}{2(\gamma-1)}} \partial_+ z \nn\\
 &= k_1\bigl(k_2(3\alpha-{\TS\frac{\gamma-1}{\gamma}}m_x z
a^{-\frac{\gamma-1}{2}})+(-{\TS\frac{\gamma-1}{\gamma}}m_x z
a^{-\frac{\gamma-1}{2}}\alpha-\alpha^2)\bigr)z^{\frac{\gamma+1}
{2(\gamma-1)}}
\nn\\\label{duct y eqn}
&\qquad{}+k_3^+(\alpha+{\TS\frac{\gamma-1}{\gamma}}m_x
z a^{-\frac{\gamma-1}{2}})
z^{\frac{\gamma+1}{2(\gamma-1)}}+F(x,t)
z^{\frac{\gamma+1}{2(\gamma-1)}}.
\end{align}

\begin{lemma}
There are variables $Y$ and $Q$, defined by
\begin{align*}
  Y &= z^{\frac{\gamma+1}{2(\gamma-1)}}\,\alpha 
     + \tilde Y (z,u,m,m_x,a,\dot a), \\
  Q &= z^{\frac{\gamma+1}{2(\gamma-1)}}\,\beta 
     + \tilde Q (z,u,m,m_x,a,\dot a),
\end{align*}
which satisfy the equations
\begin{align}
  \pp Y &= d_0 + d_1\,Y + d_2\,Y^2,  \nn\\
  \pn Q &= \bar d_0 - \bar d_1\,Q + d_2\,Q^2,
\label{YQeq}
\end{align}
with coefficients
\begin{align*}
  d_2 &= - {\TS\frac{\gamma+1}{\gamma-1}}\,K_c\,
	z^{\frac{3-\gamma}{2(\gamma-1)}},\\
  d_1&,\ \bar d_1 = d_1,\ \bar d_1(z,u,m,m_x,a,\dot a), \com{and}\\
  d_0&,\ \bar d_0 = d_0,\ \bar d_0(z,u,m,m_x,m_{xx},a,\dot a,\ddot a).
\end{align*}
\end{lemma}

\begin{proof}
For convenience, we take $\gamma\neq 5/3$ or $3$.  Writing
\eq{duct y eqn} as a derivative along the characteristic, we get
\begin{align*}
\partial_+(&z^{\frac{\gamma+1}{2(\gamma-1)}}\alpha)+
\partial_+\bigl(\frac{\gamma-1}{\gamma(3\gamma-1)}m_x 
a^{-\frac{\gamma-1}{2}}z^{\frac{3\gamma-1}
{2(\gamma-1)}}\bigr)
\\
&\ -\frac{3(\gamma-1)^2 }
{8K_c }m\, \dot{a}\,a^{-\frac{\gamma+1}{2}}\, 
z^{\frac{\gamma-3}{2(\gamma-1)}} \partial_+ z
+\partial_+\bigl(\frac{(\gamma-1)^2}{2K_c (\gamma-3)}u\,a^{-1}\, 
\dot{a}\, z^{\frac{\gamma-3}{2(\gamma-1)}} \bigr)
\\
{}={}&\frac{(\gamma-1)^2}{2K_c (\gamma-3)}\,a^{-1}\, \dot{a}\, 
z^{\frac{\gamma-3}{2(\gamma-1)}} \partial_+ u
\\
&\ +\partial_+\bigl(\frac{\gamma-1}{\gamma(3\gamma-1)}m_x 
a^{-\frac{\gamma-1}{2}}\bigr)z^{\frac{3\gamma-1}{2(\gamma-1)}}
+\frac{(\gamma-1)^2}{2K_c (\gamma-3)}u \, z^{\frac{\gamma-3}{2(\gamma-1)}} 
\partial_+ (a^{-1}\, \dot{a})
\\
&\ +k_1\bigl(k_2(3\alpha-\frac{\gamma-1}{\gamma}m_x z
a^{-\frac{\gamma-1}{2}})+(-\frac{\gamma-1}{\gamma}m_x z
a^{-\frac{\gamma-1}{2}}\alpha-\alpha^2)\bigr)z^{\frac{\gamma+1}
{2(\gamma-1)}}
\\
&\ +k_3^+(\alpha+\frac{\gamma-1}{\gamma}m_x
z a^{-\frac{\gamma-1}{2}})
z^{\frac{\gamma+1}{2(\gamma-1)}}+F(x,t)z^{\frac{\gamma+1}
{2(\gamma-1)}}.
\end{align*}
Next, we eliminate the $\pp u$ term, which comes from differentiating
$k_3^+$.  By~(\ref{duct zm}), we have
\[
  \pp u = u_t+C u_x
   = -m a^{-\frac{\gamma-1}{2}}\partial_+ z-\TS{\frac{\gamma-1}{\gamma}}C\,
z\, m_x\, a^{-\frac{\gamma-1}{2}}+\gamma\, p\, a_x,
\]
and we again write this as a total derivative plus zero-order terms in
$z$ and $u$.

Finally, we define
\begin{align}
Y &= z^{\frac{\gamma+1}{2(\gamma-1)}}\,\alpha+\TS\frac{(\gamma-1)^2}{2 
K_c(\gamma-3)}\, u\, a^{-1}\,\dot{a}\,z^{\frac{\gamma-3}{2(\gamma-1)}}
+\TS\frac{\gamma-1}{\gamma(3\gamma-1)}\,m_x\, a^{-\frac{\gamma-1}{2}}\,
z^{\frac{3\gamma-1}{2(\gamma-1)}}\nn\\&
\qquad{}-\TS\frac{(3\gamma-13)(\gamma-1)^3}{4
K_c(\gamma-3)(3\gamma-5)}\,m\,\dot{a}\,a^{-\frac{\gamma+1}{2}}\,
z^{\frac{3\gamma-5}{2(\gamma-1)}},
\label{y_def}
\end{align}
and simplify, to get
\begin{align*}
\pp Y &= -z^{\frac{3\gamma-5}{2(\gamma-1)}}\partial_+
\bigl(\TS\frac{(3\gamma-13)(\gamma-1)^3}{4K_c(\gamma-3)
(3\gamma-5)}m\, \dot{a}\,a^{-\frac{\gamma+1}{2}}\bigr) \\&{}
+{}\TS{\frac{(\gamma-1)^2}{2K_c (\gamma-3)}\,a^{-1}\, \dot{a}\, 
z^{\frac{\gamma-3}{2(\gamma-1)}}\bigl(-\frac{\gamma-1}{\gamma}C\,
z\, m_x\, a^{-\frac{\gamma-1}{2}}+\gamma\, p\, a_x\bigr)}\\&{}
+{}\TS{\bigl(\frac{\gamma-1}{\gamma(3\gamma-1)}\,m_x\, 
a^{-\frac{\gamma-1}{2}}\bigr)z^{\frac{3\gamma-1}{2(\gamma-1)}}
+\frac{(\gamma-1)^2}{2K_c (\gamma-3)}u \, z^{\frac{\gamma-3}{2(\gamma-1)}} 
\partial_+ (a^{-1}\, \dot{a})}\\&{}
+{}k_1\TS{\bigl(k_2(3\alpha-\frac{\gamma-1}{\gamma}m_x z
a^{-\frac{\gamma-1}{2}})+(-\frac{\gamma-1}{\gamma}m_x z
a^{-\frac{\gamma-1}{2}}\alpha-\alpha^2)\bigr)z^{\frac{\gamma+1}
{2(\gamma-1)}}}\\&{}+{}\TS{
k_3^+(\alpha+\frac{\gamma-1}{\gamma}m_x
z a^{-\frac{\gamma-1}{2}})
z^{\frac{\gamma+1}{2(\gamma-1)}}+F(x,t)z^{\frac{\gamma+1}
{2(\gamma-1)}}}.
\end{align*}

Using \eq{y_def} to write
\[
  Y = z^{\frac{\gamma+1}{2(\gamma-1)}}\,\alpha + y_0, \com{so}
  \alpha = z^{-\frac{\gamma+1}{2(\gamma-1)}}\,Y 
         - z^{-\frac{\gamma+1}{2(\gamma-1)}}\,y_0,
\]
we eliminate $\alpha$ to get
\[
  \pp Y=d_0+d_1\,Y+d_2\,Y^2,
\]
with coefficients $d_i$ containing no derivatives of $z$, $u$, and in
particular,
\[
  d_2 = - \TS\frac{\gamma+1}{\gamma-1}\,K_c\,z^{\frac{3-\gamma}{2(\gamma-1)}}.
\]
Using $z^{\frac2{\gamma-1}} = K_\tau\,a\,\rho$, after a tedious
calculation we get, for $\gamma \ne 5/3$ or 3,
\begin{align*}
d_0 &= a^\frac{\gamma-3}{4}\big(\TS \frac{1}{\gamma}
K_c K_\tau^\frac{5\gamma+1}{4}(\TS\frac{\gamma-1}{3\gamma-1}m 
m_{xx}-\TS\frac{(3\gamma+1)(\gamma-1)}{(3\gamma-1)^2}m_x^2 )\,
a^2\, \rho^{\frac{5\gamma+1}{4}}\\
&\quad{}+L_1\,m_x\,\dot{a}\,u\,\rho^{\frac{3\gamma-1}{4}}
+L_2\,m\,m_x\,\dot{a}\,\rho^{\frac{5\gamma-3}{4}}\\
&\quad{}+L_3\,m\,\dot{a}^2\,a^{-2}\,u\,
\rho^{\frac{3\gamma-5}{4}}
+L_4\,m\,\ddot{a}\,a^{-1}\,u\,\rho^{\frac{3\gamma-5}{4}}
\\
&\quad{}+L_5\,m^2\,a^{-2}\,\dot{a}^2\,\rho^{\frac{5\gamma-7}{4}}
+L_6\,\ddot{a}\,a^{-1}\,u^2\,\rho^{\frac{\gamma-3}{4}}\\
&\quad{}+L_7\,a^{-2}\,\dot{a}^2\,u^2\,\rho^{\frac{\gamma-3}{4}}
+L_8\,m^2\,\ddot{a}\,a^{-1}\,\rho^{\frac{5\gamma-7}{4}}\big),\\
d_1 &= L_9\, m_x\, a\, \rho^{\frac{\gamma+1}{2}}+
L_{10}\,u\,a^{-1}\,\dot{a}+L_{11}\, m\, a^{-1}\, \dot{a}\, \rho^{\frac{\gamma-1}{2}},\\
d_2 &= -\TS\frac{\gamma+1}{\gamma-1}K_c\, K_\tau^\frac{3-\gamma}{4} 
\, a^\frac{3-\gamma}{4}\,\rho^{\frac{3-\gamma}{4}}<0,
\end{align*}
where the $L_j$ are constants depending only on $\gamma$.  

Similarly, if $Q$ is defined by
\begin{align*}
Q&=z^{\frac{\gamma+1}{2(\gamma-1)}}\,\alpha+\TS\frac{(\gamma-1)^2}{2 
K_c(\gamma-3)}\, u\, a^{-1}\,\dot{a}\,z^{\frac{\gamma-3}{2(\gamma-1)}}
-\TS\frac{\gamma-1}{\gamma(3\gamma-1)}\,m_x\, a^{-\frac{\gamma-1}{2}}\,
z^{\frac{3\gamma-1}{2(\gamma-1)}}
\\&\quad{}+{}\TS\frac{(3\gamma-13)(\gamma-1)^3}{4
K_c(\gamma-3)(3\gamma-5)}\,m\, \dot{a}\,a^{-\frac{\gamma+1}{2}}\,
z^{\frac{3\gamma-5}{2(\gamma-1)}},
\end{align*}
then it satisfies
\[
  \pn Q= {\bar d}_0+{\bar d}_1\, Q+d_2\, Q^2,
\]
with, for $\gamma\neq 5/3$ or $3$,
\begin{align*}
{\bar d}_0&=a^\frac{\gamma-3}{4}\big(\TS \frac{1}{\gamma}
K_c K_\tau^\frac{5\gamma+1}{4}(\TS\frac{\gamma-1}{3\gamma-1}m 
m_{xx}-\TS\frac{(3\gamma+1)(\gamma-1)}{(3\gamma-1)^2}m_x^2 )\,
a^2\, \rho^{\frac{5\gamma+1}{4}}\\
&\quad{}+\bar L_1\,m_x\,\dot{a}\,u\,\rho^{\frac{3\gamma-1}{4}}
+\bar L_2\,m\,m_x\,\dot{a}\,\rho^{\frac{5\gamma-3}{4}}\\
&\quad{}+\bar L_3\,m\,\dot{a}^2\,a^{-2}\,u\,
\rho^{\frac{3\gamma-5}{4}}
+\bar L_4\,m\,\ddot{a}\,a^{-1}\,u\,\rho^{\frac{3\gamma-5}{4}}
\\
&\quad{}+\bar L_5\,m^2\,a^{-2}\,\dot{a}^2\,\rho^{\frac{5\gamma-7}{4}}
+\bar L_6\,\ddot{a}\,a^{-1}\,u^2\,\rho^{\frac{\gamma-3}{4}}\\
&\quad{}+\bar L_7\,a^{-2}\,\dot{a}^2\,u^2\,\rho^{\frac{\gamma-3}{4}}
+\bar L_8\,m^2\,\ddot{a}\,a^{-1}\,\rho^{\frac{5\gamma-7}{4}}\big),\\
{\bar d}_1&=\bar L_9\, m_x\, a\, \rho^{\frac{\gamma+1}{2}}+
\bar L_{10}\,u\,a^{-1}\,\dot{a}+\bar L_{11}\, m\, a^{-1}\, \dot{a}\,
\rho^{\frac{\gamma-1}{2}},
\end{align*}
and $\bar L_j$ are constants depending only on $\gamma$.  It is clear
that similar equations hold for $\gamma=5/3$ or 3.
\end{proof}

\subsection{Singularity formation}

Having obtained Riccati equations \eq{YQeq} for derivative variables,
we obtain a gradient blowup result as in Theorem \ref{Thm
  singularity2}.  The result states that a strong enough compressive
wave will eventually form a shock, provided the state variables remain
finite.  Theorem~\ref{Thm_upper} indicates that state variables should
indeed remain bounded for all times, although a vacuum can form in
infinite time~\cite{YpRP2,G1}.

\begin{cor}
\label{duct Thm singularity2}
Assume there exist some positive constants $M_i$, such that  
the duct and initial data satisfy
\begin{align*}
  M_1<a(x,0)<M_2,\quad& |\dot{a}(x,0)|<M_3,\quad |\ddot{a}(x,0)|<M_4,\\
  M_5<m(x,0)<M_6,\quad& |m_x(x,0)|<M_7,\quad |m_{xx}(x,0)|<M_8,
\end{align*}
while the solutions satisfy the \emph{a priori} bounds
\[
  |u(x,t)|<M_{9}, \quad \rho(x,t)<M_{10},
\]
and $\rho(x,t)>M_{11}$ if $1<\gamma \leq 3$, for 
$(x,t)\in \RR\times\RR^+$.  Then there exists a positive constant
$\bar{N}$, depending on $\gamma$ and $M_i$, such that, if
\[
  Y < -\bar{N}  \com{or}  Q < -\bar N
\]
somewhere in the initial data, then a singularity forms in finite time.
\end{cor}

\begin{proof}
The roots of (\ref{YQeq}), if they exist, are given by the quadratic
formula
\[
  \frac{-d_1\pm\sqrt{d^2_1-4\,d_0\,d_2}}{2\,d_2} \com{or}
  \frac{\bar d_1\pm\sqrt{{\bar d}_1^2-4\,\bar d_0\,d_2}}{2\,d_2},
\]
which have bounds depending on the ratios $|d_1|/|d_2|$, $|d_0|/|d_2|$,
$|\bar d_1|/|d_2|$ and $|\bar d_0|/|d_2|$.  It follows that these
roots have a lower bound $-\bar N$ depending on the $M_i$, and the
proof follows in the same way as Theorem~\ref{Thm singularity2}.
\end{proof}

\subsection{Spherically symmetric Euler equations}

The equations for compressible flow in a duct include spherically
symmetric three-dimensional Euler flow as a special case, on the
exterior of a ball.  Using $r=|x'|$ as the radial coordinate, the
Euler equations reduce to the duct equations with $x$ replaced by $r$
and $a(r)=r^2$.  

In order to get cleaner equations, we consider the isentropic case,
with $m=1$, say, and $\gamma\neq 5/3$ or $3$.  With these assumptions,
$m_x=0$, $\dot a = 2\,r$ and $\ddot a = 2$, and (\ref{YQeq}) hold with
coefficients
\begin{align*}
d_0&=r^{\frac{\gamma-3}2-2}\, \rho^{\frac{\gamma-3}{4}}\big(G_1\,
u\, \rho^{\frac{\gamma-1}{2}} +G_2\, \rho^{\gamma-1} 
+G_3\, u^2 \big)\\ 
d_1&=r^{-1}\big(G_4\, u + G_5\,
\rho^{\frac{\gamma-1}{2}}\big)\\ 
d_2&=-\TS\frac{\gamma+1}{\gamma-1}K_c\,
K_\tau^\frac{3-\gamma}{4}\,
r^\frac{3-\gamma}{2}\,\rho^{\frac{3-\gamma}{4}},
\end{align*}
for constants $G_i$.  Thus the ratios are
\begin{align*}
  \frac{d_1}{d_2} &= r^{\frac{\gamma-5}2}\,\rho^{\frac{\gamma-3}{4}}\,
    (\hat G_4\, u + \hat G_5\,
\rho^{\frac{\gamma-1}{2}}) \\
  \frac{d_0}{d_2} &= r^{\gamma-5}\,\rho^{\frac{\gamma-3}{2}}\,
    (\hat G_1\,u\, \rho^{\frac{\gamma-1}{2}} +\hat G_2\,
    \rho^{\gamma-1}+\hat G_3\, u^2),
\end{align*}
for constants $\hat G_i$.  It follows that the roots
$\beta_\pm = \frac{-d_1\pm\sqrt{d^2_1-4d_0 d_2}}{2d_2}$ of (\ref{YQeq})
satisfy the bound
\[
  \beta_\pm \ge - r^{\frac{\gamma-5}2}\,\rho^{\frac{\gamma-3}{4}}\,(
        \hat G_6\,|u| + \hat G_7\,\rho^{\frac{\gamma-1}{2}}),
\]
for constants $\hat G_j$ depending only on the equation of state.

A gradient blowup result similar to Theorem~\ref{Thm singularity2} now
follows exactly as before for the {\em exterior} problem when
$1<\gamma<5$.

\end{document}